\documentclass[psamsfonts]{amsart}

\usepackage{amssymb,amsfonts}
\usepackage[all,arc]{xy}
\usepackage{enumerate}
\usepackage{mathrsfs}
\newtheorem{thm}{Theorem}[section]
\newtheorem{cor}[thm]{Corollary}
\newtheorem{prop}[thm]{Proposition}

\theoremstyle{definition}
\newtheorem{defn}[thm]{Definition}

\newtheorem{exmp}[thm]{Example}

\theoremstyle{remark}

\makeatletter
\let\c@equation\c@thm
\makeatother
\numberwithin{equation}{section}

\title{\small{Statistical quasi Cauchyness in two normed spaces}}

\author{\; \; \;Huseyin Cakalli \& Sibel Ersan\\
Maltepe University, TR 34857, \.{I}stanbul-Turkey}

\address{H\"usey\.{I}n \c{C}akall\i \\
          Maltepe University, Department of Mathematics, Marmara E\u{g}\.{I}t\.{I}m K\"oy\"u, TR 34857, Maltepe, \.{I}stanbul-Turkey \; \; \; \; \; Phone:(+90216)6261050 ext:2248, \;  fax:(+90216)6261113}

\email{huseyincakalli@maltepe.edu.tr; hcakalli@gmail.com}

\address{Sibel Ersan\\
            Faculty of Engineering And Natural Sciences, Maltepe University, Marmara E\u{g}\.{I}t\.{I}m K\"oy\"u, TR 34857, Maltepe, \.{I}stanbul-Turkey\\ Phone: (+90216) 6261050 ext:2396, fax: (+90216) 6261131}

\email{sibelersan@maltepe.edu.tr; sibelersan@gmail.com}
\date{\today}

\keywords{Statistical convergence, quasi-Cauchy sequences, continuity}
\subjclass[2010]{Primary: 40A05 Secondaries:40A35, 26A15, 40A30}

\begin{document}

\begin{abstract}
A function $f$ defined on a subset $E$ of a two normed space $X$ is statistically ward continuous if it preserves statistically quasi-Cauchy sequences of points in $E$ where a sequence $(x_n)$ is statistically quasi-Cauchy if $(\Delta x_{n})$ is a statistically null sequence. A subset $E$ of $X$ is statistically ward compact if any sequence of points in $E$ has a statistically quasi-Cauchy subsequence. In this paper, new kinds of continuities are investigated in two normed spaces. It turns out that uniform limit of statistically ward continuous functions is again statistically ward continuous.

\end{abstract}

\maketitle

\section{Introduction}
 The concept of continuity and any concept involving continuity play a very important role in pure mathematics and also in other branches of sciences involving mathematics, i.e. in computer science, information theory, and biological science.

The idea of statistical convergence was first given under the name "almost convergence" by Zygmund in Warsaw in 1935 \cite{ZygmundTrigonometricseries}. Statistical convergence was formally introduced by Fast \cite{Fast} and later was reintroduced by Schoenberg \cite{Schoenberg}, and also independently by Buck \cite{BuckGeneralizedasymptoticdensity}. This concept has become an active area of active research. It has been applied in various areas (\cite{ErdosSurlesdensitesde,MillerAmeasuretheoresubsequencecharacterizationofstatisticalconvergence,Schoenberg,FridyOnstatisticalconvergence,CasertaMaioKocinacStatisticalConvergenceinFunctionSpaces,ZygmundTrigonometricseries,FreedmanandSemberDensitiesandsummability,MaddoxStatisticalconvergenceinlocallyconvex,ConnorandSwardsonStrongintegralsummabilityandstonecompactification,MakarovLevinRubinovMathematicalEconomicTheory}).

A new development of the concept of $2$-normed spaces began with 1962 by S. G\"ahler (\cite{gohler}, \cite{g2}, and \cite{g3}). Since then, this concept has been studied by many authors (\cite{gunawan, aleks,ist}). They contributed a lot for the extension of this branch of mathematics.

A real function $f$ is continuous if and only if it preserves convergent sequences. A subset $E$ of $\textbf{R}$, the set of real numbers, is compact if any sequence of points in $E$ has a convergent subsequence whose limit is in $E$. Using the idea of continuity and compactness, many kinds of continuities and compactness were introduced and investigated (\cite{burton,cakalli,cakallislowlyoscillatingcontinuity,CakalliNewkindsofcontinuities, CakallistatisticalquasiCauchysequences, DikandCanak}). Statistical ward continuity of a real function and statistical ward compactness of a subset $E$ of $\textbf{R}$ were introduced by Cakalli in \cite{CakalliStatisticalwardcontinuity} in the sense that a real function $f$ is called statistically ward continuous on $E$ if the sequence $\left(f\left(x_n\right)\right)$ is statistically quasi-Cauchy whenever $\textbf{x}=(x_n)$ is a statistically quasi-Cauchy sequence of points in $E$.

The aim of this paper is to investigate statistically ward continuity in two normed spaces, and prove interesting theorems.
\section{Preliminaries}

First of all, some definitions and notation will be given in the following. Throughout this paper, $\textbf{N}$, and $\textbf{R}$ will denote the set of all positive integers, and the set of all real numbers, respectively. First we recall the definition of a two normed space.
\begin{defn} \label{Definitionofatwonormedspace} (\cite{gohler})
\emph{Let $X$ be a real linear space with $dim X >1$ and \; \; \; \; \; $||.,.||:X^{2}\rightarrow \textbf{R}$ a function. Then $(X,||.,.||)$ is called a linear $2$-normed space if
\begin{enumerate}
	\item $\left\|x,y\right\|=0\Leftrightarrow$ x and y are linearly dependent,
	\item $\left\|x,y\right\|=\left\|y,x\right\|$,
	\item $\left\|\alpha x,y\right\|=\left|\alpha\right|\left\|x,y\right\|$,
	\item $\left\|x,y+z\right\|\leq\left\|x,y\right\|+\left\|x,z\right\|$
\end{enumerate}
for $\alpha \in \textbf{R}$ and $x,y,z\in X$. The function $||.,.||$ is called the $2$-norm on $X$.
}

Observe that in any $2$-normed space $ \left(X,\left\|.,.\right\|\right)$ we have $ \left\|.,.\right\|$ is nonnegative, $||x-z,x-y||=||x-z,y-z||$, and $\forall x,y\in X, \alpha\in \mathbf{R}$ $\left\|x,y+\alpha x\right\|=\left\|x,y\right\|$. Throughout this paper by $X$ we will mean a $2$-normed space with a two norm $\left\|.,.\right\|$.
\end{defn}
A classical example is the two normed space $X=\mathbf{R}^{2}$ with the two norm $\left\|.,.\right\|$ defined  by $\left\|a,b\right\|=\left|a_{1}b_{2}-a_{2}b_{1}\right|$ where $a=\left(a_{1},a_{2}\right)$, $b=\left(b_{1},b_{2}\right)\in \mathbf{R}^{2}$. This is
the area of the parallelogram determined by the vectors $a$ and $b$.

A sequence $(x_{n})$ of points in $X$ is said to be convergent to an element $x\in{X}$ if $lim_{n\rightarrow\infty}\left\|x_{n}-x,z\right\|=0$ for every $z\in X$. This is denoted by $lim_{n\rightarrow\infty} ||x_{n}, z|| = ||x,z||$. A sequence $\left(x_n\right)$  of points in $X$ is called Cauchy if  $lim_{n,m\rightarrow\infty}\left\|x_{n}-x_{m},z\right\|=0$ for every $z\in X$ (\cite{gohler}).

A sequence of functions $\left(f_n\right)$ is said to be uniformly convergent to a function $f$ on a subset $E$ of $X$ if for each $\epsilon>0$, an integer $N$ can be found such that $\left\|f_n\left(x\right)-f\left(x\right),z\right\|<\epsilon$ for $n\geq N$ and for all $x,z\in X$ (\cite{gohler}).

The concept of statistical convergence is a generalization of the usual notion of convergence that, for real-valued sequences, parallels the usual theory of convergence. For a subset $M$ of $\textbf{N}$ the asymptotic density of $M,$ denoted by $\delta(M)$, is given by
\[
\delta(M)=\lim_{n\rightarrow\infty}\frac{1}{n}|\{k\leq n: k\in M\}|,
\]
if this limit exists where $|\{k \leq n: k\in {M}\}|$ denotes the cardinality of the set $\{k \leq n : k \in{M}\}$.
Let recall that a sequence $(x_k)$ is said to be statistically convergent to $L$ if for every $\epsilon>0$ the set $\{k\in N: ||x_k-L,z||\geq\epsilon\}$ has natural density zero for each nonzero $z$ in $X$, in other words $(x_k)$ statistically converges to $L$ in 2-normed space $X$ if $$\lim_{k\rightarrow\infty}\frac{1}{k}|\{k\in N: ||x_k-L,z||\geq\epsilon\}|=0$$ for each nonzero $z$ in $X$. For $L=0$, we say this is statistically null \cite{GurdalandPehlivanStatisticalconvergencein$2$-normedspaces}.

\section{Results}

In this section we investigate the notion of statistically ward continuity of a function on a two normed space. First we give a definition of a statistically quasi-Cauchy sequence in a two normed space.
\begin{defn} \textit{A sequence $(x_k)$ of points in $X$ statistically quasi-Cauchy if $st-lim_{k\rightarrow\infty}\Delta x_k=0$ where $\Delta x_{k}=x_{k+1}-x_{k}$. That is, for each $\epsilon$,
\begin{equation*}
lim_{n\rightarrow\infty}\frac{1}{n}\left|\left\{k\leq n:\left\|\Delta x_k,z\right\|\geq\epsilon\right\}\right|=0, \forall z\in X.
\end{equation*}}
\end{defn}
We note that any quasi-Cauchy sequence is statistically quasi-Cauchy not only in the real case and in the metric setting, but also in a $2$-normed space. A statistically convergent sequence is statistically quasi-Cauchy. However  the converse is not always true, i.e. there are statistically quasi-Cauchy sequences which are not statistically convergent. Any Cauchy sequence is statistically quasi-Cauchy, but the converse is not always true.

Now we give the definition of statistical ward compactness of a subset of $X$.
\begin{defn}
\textit{A subset $E$ of $X$ is called statistically ward compact if any sequence of points in $E$ has a statistically quasi-Cauchy subsequence.}
\end{defn}

First, we note that any finite subset of $X$ is statistically ward compact, union of two statistically ward compact subsets of $X$ is statistically ward compact and intersection of any statistically ward compact subsets of $X$ is statistically ward compact. Furthermore any subset of a statistically ward compact set is statistically ward compact. Any compact subset of $X$ is also statistically ward compact.

A function $f$ on a subset $E$ of $X$ is sequentially continuous at $x_0$ if for any sequence $(x_{n})$  of points in $E$ converging to $x_{0}$, we have $(f(x_{n}))$ converges to $f(x_{0})$. $f$ is sequentially continuous on $E$ if it is sequentially continuous at every point of $E$. This is equivalent to the statement that $f$ preserves convergent sequences of points in $E$.
\begin{defn}
\textit{A function $f$ on a subset $E$ of $X$ is said to be statistically sequentially continuous at $x_0$ if for any sequence $(x_{n})$  of points in $E$ statistically converging to $x_{0}$, we have $(f(x_{n}))$ statistically converges to $f(x_{0})$ (see also \cite{MaioKocinac} and \cite{CakalliandKhan}).}
\end{defn}
\begin{thm} \label{Thmstatisticallysequentiallycontinuityimpliessequentialcontinuity} \textit{If $f$ is statistically sequentially continuous, then it is sequentially continuous.}
\end{thm}
\begin{proof} Suppose that $f$ is not sequentially continuous at a point $x$ of $X$ so that there exists a convergent sequence $\textbf{x}=(x_{n})$  of points in $E$ with limit $x$ such that $(f(x_n))$ is not convergent to $f(x)$. Then there exists a $z\in X$, and a positive real number $\epsilon_0$ such that for each $n\in \textbf{N}$ there is a $k_n\in N$ with $||f(x_{k_{n}})-f(x),z||\geq \epsilon_0$. Consider the subsequence $(x_{k_{n}})$ of $\textbf{x}$. As statistical sequential method is regular, $(x_{k_{n}})$ is statistically convergent to $x$. Since $$\{k\leq n: ||f(x_{k_{n}})-f(x),z||\geq \epsilon_0\}=\{1,2,...,k\},$$ the transformed sequence $(f(x_n))$ is not statistically convergent to $f(x)$. This completes the proof of the theorem.
\end{proof}
\begin{thm} \textit{If $st-\lim_{n\rightarrow\infty}x_n=x_0$ implies that $\lim_{n\rightarrow\infty}f(x_n)=f(x_0)$ for an $x_0\in X$, then $f$ is a constant function.}
\end{thm}
\begin{proof}
Let $x_0$ be a fixed element of $X$. We can construct a sequence as
$$
 x_{n} =
 \begin{cases}
 t \;, & \text{if }n=k^2\text{ for a positive integer}\; k \\
 x_0 \;, & \text{if otherwise}
 \end{cases}
 .$$
It is easy to see that $st-\lim_{n\rightarrow\infty}x_n=x_0$. The transformed sequence $(f(x_n))$ defined by
$$
 f(x_{n}) =
 \begin{cases}
 f(t) \;, & \text{if }n=k^2\text{ for a positive integer}\; k \\
 f(x_0) \;, & \text{if otherwise}
 \end{cases}
 $$
is convergent. Thus the subsequence $$(w_n)=(f(x_{n^2}))=(f(x_{1}),f(x_{4}),f(x_{9}),...,f(x_{n^2}),...)$$ of the sequence $(f(x_n))$ is also convergent to $f(x_0)$. Since for all $n\in \textbf{N}$, $w_n=f(x_{n^2})=f(t)$, it follows that $f(t)=f(x_0)$. This completes the proof.
\end{proof}
\begin{defn}
\textit{A function defined on a subset $E$ of $X$ is called statistically ward continuous if it preserves statistically quasi-Cauchy sequence, i.e. $(f(x_n))$ is statistically quasi-Cauchy whenever $(x_n)$ is.}
\end{defn}
We note that a composite of two statistically ward continuous functions is statistically ward continuous. Now we prove that sum of two statistically ward continuous functions is statistically ward continuous.
\begin{prop}
\textit{Sum of two statistically ward continuous functions is statistically ward continuous.}
\end{prop}
\begin{proof}
Let $f$ and $g$ be statistically ward continuous functions on a subset $E$ of $X$. To prove that $f+g$ is statistically ward continuous on $E$, take any statistically quasi-Cauchy sequence $(x_{k})$ in $E$. Then $(f(x_{k}))$ and $(g(x_{k}))$ are statistically quasi-Cauchy sequences. Let $\varepsilon>0$ be given. Since $(f(x_{k}))$ and $(g(x_{k}))$ are statistically quasi-Cauchy, we have
\[
\lim_{n\rightarrow\infty}\frac{1}{n }|\{k\leq n: |\Delta (f(x_{k}))| \geq\frac{\varepsilon}{2}\}|=0
\]
and
\[
\lim_{n\rightarrow\infty}\frac{1}{n }|\{k\leq n: |\Delta (g(x_{k}))| \geq\frac{\varepsilon}{2}\}|=0
.\]
Hence \[
\lim_{n\rightarrow\infty}\frac{1}{n }|\{k\leq n: |\Delta (f(x_{k})+g(x_{k}))| \geq \varepsilon\}|=0
\]
which follows from the following inclusion
\begin{displaymath}
\{k\in I_{n}: |\Delta ((f+g)(x_{k}))| \geq\frac{\varepsilon}{2}\}\subset{\{k\leq n: |\Delta (f(\alpha_{k}))| \geq\frac{\varepsilon}{2}\} \cup {\{k\leq n: |\Delta (g(x_{k}))| \geq\frac{\varepsilon}{2}\}}}.
\end{displaymath}

\end{proof}
In connection with statistically quasi-Cauchy sequences and convergent sequences the problem arises to investigate the following types of continuity of functions on $X$:
\begin{enumerate}
	\item $\left(x_n\right)$ is statistical quasi-Cauchy $\Rightarrow\left(f\left(x_n\right)\right)$ is statistical quasi-Cauchy.
	\item $\left(x_n\right)$ is statistical quasi-Cauchy $\Rightarrow\left(f\left(x_n\right)\right)$ is convergent.
	\item $\left(x_n\right)$ is convergent $\Rightarrow\left(f\left(x_n\right)\right)$ is convergent.
	\item $\left(x_n\right)$ is convergent $\Rightarrow\left(f\left(x_n\right)\right)$ is statistical quasi-Cauchy.
	\item $\left(x_n\right)$ is statistical convergent $\Rightarrow\left(f\left(x_n\right)\right)$ is statistical convergent.
\end{enumerate}

In the previous statements (1) is a statistical ward continuity, (3) is the ordinary continuity and (5) is a statistical continuity of the function $f$. It is obvious that $(2)\Rightarrow (1)$ while $(1)$ does not imply $(2)$, $(1)\Rightarrow (4)$ while $(4)$ does not imply $(1)$, $(2)\Rightarrow (3)$ while $(3)$ does not imply $(2)$ and lastly $(3)$ is equivalent to $(4)$. We give the definition of statistical sequential continuity in $2$-normed spaces in the following before proving that (1) implies (5).

\begin{thm}\label{43}
 \textit{If $f:X\rightarrow X$ is statistically ward continuous on a subset $E$ of $X$, then it is statistically sequentially continuous on $E$.}
\end{thm}
\begin{proof}
Let $(x_k)$ be any statistically convergent sequence of points in E with a statistically limit $x_{0}$. Hence for all $z\in{X}$
$$lim_{n\rightarrow\infty}\frac{1}{n}\left|\left\{k\leq n:\left\|x_k-x_0,z\right\|\geq\epsilon\right\}\right|=0.$$
Then the sequence $\boldsymbol{\xi}=(\xi_{n})$ defined by
$$\xi_{n} =
 \begin{cases}
 x_{k} \;, & \text{if }n=2k-1\text{ for a positive integer}\; k \\
 x_{0} \;, & \text{if }n\text{ is even}
 \end{cases}
 $$
is also statistically convergent to $x_{0}$. Therefore it is a statistically quasi-Cauchy sequence. As $f$ is statistically ward continuous on $E$, the transformed sequence $f(\boldsymbol{\xi})=(f(\xi_{n}))$ obtained by $$f(\xi_{n}) =
 \begin{cases}
 f(x_{k}) \;, & \text{if }n=2k-1\text{ for a positive integer}\; k \\
 f(x_{0}) \;, & \text{if }n\text{ is even}
 \end{cases}
 $$
is also statistically quasi-Cauchy. Now it follows that $$lim_{n\rightarrow\infty}\frac{1}{n}\left|\left\{k\leq n:\left\|f(x_k)-f(x_0),z\right\|\geq\epsilon\right\}\right|=0$$ for every $z\in{X}$. It implies that the sequence $\left(f\left(x_n\right)\right)$  statistically converges to $f\left(x_0\right)$. This completes the proof of the theorem.
\end{proof}
The converse of the theorem is not always true. The following example illustrates this situation:
\begin{exmp} Consider two normed space $\textbf{R}^{2}$ with the two norm $\left\|(a_1,a_2),(b_1,b_2)\right\|=\left|a_{1}b_{2}-a_{2}b_{1}\right|$, the sequence $x=(x_k)=(x_{k}^1,x_{k}^2)=(\sqrt{k},\sqrt{k})$, and the function $f(x)=f(x_{1}, x_{2})=((x_{1})^{2}, (x_{2})^{2})$. Thus
it is easily verified that the sequence $(x_k)$ is statistically quasi-Cauchy since $\forall z\in X$
\begin{align*}
&lim_{n\rightarrow\infty}\frac{1}{n}\left|\left\{k\leq n:\left\|\Delta x_k,z\right\|\geq\epsilon\right\}\right|
=lim_{n\rightarrow\infty}\frac{1}{n}\left|\left\{k\leq n:\left\|x_{k+1}-x_k,z\right\|\geq\epsilon\right\}\right|\\
&=lim_{n\rightarrow\infty}\frac{1}{n}\left|\left\{k\leq n:\left\|\left(\sqrt{k+1},\sqrt{k+1}\right)-\left(\sqrt{k},\sqrt{k}\right),z\right\|\geq\epsilon\right\}\right|\\
&=lim_{n\rightarrow\infty}\frac{1}{n}\left|\left\{k\leq n:\left\| \left(\sqrt{k+1}-\sqrt{k},\sqrt{k+1}-\sqrt{k}\right),z\right\|\geq\epsilon\right\}\right|\\
&=0
\end{align*}
On the other hand
$\left(f\left(x_k\right)\right)=\left(f\left(\sqrt{k},\sqrt{k}\right)\right)=(k,k)$ is not statistically quasi-Cauchy quasi-Cauchy. Since $\forall z\in X$
\begin{align*}
&lim_{n\rightarrow\infty}\frac{1}{n}\left|\left\{k\leq n:\left\|\Delta f\left( x_k\right),z\right\|\geq\frac{1}{2}\right\}\right|\\
&=lim_{n\rightarrow\infty}\frac{1}{n}\left|\left\{k\leq n:\left\|f\left(x_{k+1}\right)-f\left(x_k\right),z\right\|\geq\frac{1}{2}\right\}\right|\\
&=lim_{n\rightarrow\infty}\frac{1}{n}\left|\left\{k\leq n:\left\|\left(k+1,k+1\right)-\left(k,k\right),z\right\|\geq\frac{1}{2}\right\}\right|\\
&=lim_{n\rightarrow\infty}\frac{1}{n}\left|\left\{k\leq n:\left\|\left(1,1\right),\left(z_1,z_2\right)\right\|\geq\frac{1}{2}\right\}\right|\\
&=lim_{n\rightarrow\infty}\frac{1}{n}\left|\left\{k\leq n:z_2-z_1\geq\frac{1}{2}\right\}\right|,\ \text{for } z_1=1,z_2=4\\
&=lim_{n\rightarrow\infty}\frac{1}{n}\left|\left\{k\leq n:3\geq\frac{1}{2}\right\}\right|=lim_{n\rightarrow\infty}\frac{1}{n}n=1\neq 0
\end{align*}
\end{exmp}
Now we state the following result related to ordinary sequential continuity.
\begin{cor}
\textit{If a function $f$ is statistically ward continuous, then it is sequentially continuous in the ordinary sense.}
\end{cor}
\begin{proof} The proof follows from Theorem \ref{Thmstatisticallysequentiallycontinuityimpliessequentialcontinuity} so is omitted.

\end{proof}
\begin{thm}\textit{ Statistically ward continuous image of any statistically ward compact subset of $X$ is statistically ward compact.}
\end{thm}
\begin{proof}
Let $E$ be a statistically ward compact subset of $X$. Statistically ward compactness of $E$ implies that there is a subsequence $z=(z_k)$ of $\bold x=(x_n)$ with $st-lim_{k\rightarrow\infty}||\Delta z_k,y||=0, \forall y\in X.$ Then $(t_k)=(f(z_k))$. $(t_k)$ is a subsequence of the sequence $f(x)$ such that  $st-lim_{k\rightarrow\infty}||\Delta t_k,f(y)||=0.$ This completes the proof of the theorem.
\end{proof}
In ordinary sense it is well known that uniform limit of continuous functions is continuous. It is also true that uniform limit of statistically ward continuous functions is statistically ward continuous.

\begin{thm}
\textit{If a sequence $(f_n)$ of statistically ward continuous functions on a subset $E$ of $X$ and $(f_n)$ is uniformly convergent to a function $f$ then $f$ is statistically ward continuous on $E$.}
\end{thm}
\begin{proof}
Let z be any fixed point of $X$ and $\epsilon\geq0$. Take any statistically quasi-Cauchy sequence $(x_k)$ of points in $E$. By uniform convergence of $(f_n)$ there exists a positive integer $\textbf{N}$ such that $\left\|f_n(x)-f(x),z\right\|<\frac{\epsilon}{3}$, $\forall x\in E$ whenever $n\geq N$. Since the function $f_N$ is statistically ward continuous on $E$, we have
\begin{equation*}
\lim_{n\rightarrow\infty}\frac{1}{n}\left|\left\{k\leq n:\left\|f_N(x_{k+1})-f_N(x_k),z\right\|\geq\frac{\epsilon}{3}\right\}\right|=0
\end{equation*}
By the triangle inequality we have
\begin{align*}
\left\|f(x_{k+1})-f(x_k),z\right\|&\leq\left\|f(x_{k+1})-f_N(x_{k+1}),z\right\|\\
&+\left\|f_N(x_{k+1})-f_N(x_k),z\right\|\\
&+\left\|f_N(x_{k})-f(x_k),z\right\|
\end{align*}
On the other hand, we have
\begin{align*}
\left\{k\leq n:\left\|f(x_{k+1})-f(x_k),z\right\|\geq\epsilon\right\}&\subset\left\{k\leq n:\left\|f(x_{k+1})-f_N(x_{k+1}),z\right\|\geq\frac{\epsilon}{3}\right\}\\
&\cup\left\{k\leq n:\left\|f_N(x_{k+1})-f_N(x_k),z\right\|\geq\frac{\epsilon}{3}\right\}\\
&\cup\left\{k\leq n:\left\|f_N(x_{k})-f(x_k),z\right\|\geq\frac{\epsilon}{3}\right\}
\end{align*}
From this inclusion the following is obtained.
\begin{align*}
&\lim_{n\rightarrow\infty}\frac{1}{n}\left|\left\{k\leq n:\left\|f(x_{k+1})-f(x_k),z\right\|\geq\epsilon\right\}\right|\\
&\leq\lim_{n\rightarrow\infty}\frac{1}{n}\left|\left\{k\leq n:\left\|f(x_{k+1})-f_N(x_{k+1}),z\right\|\geq\frac{\epsilon}{3}\right\}\right|\\
&+\lim_{n\rightarrow\infty}\frac{1}{n}\left|\left\{k\leq n:\left\|f_N(x_{k+1})-f_N(x_k),z\right\|\geq\frac{\epsilon}{3}\right\}\right|\\
&+\lim_{n\rightarrow\infty}\frac{1}{n}\left|\left\{k\leq n:\left\|f_N(x_{k})-f(x_k),z\right\|\geq\frac{\epsilon}{3}\right\}\right|=0
\end{align*}
Thus $f$ is statistically ward continuous on $E$. This completes the proof of the theorem.
\end{proof}

\section{Conclusion}

In this paper, we introduced new kinds of continuities and compactness, and proved interesting theorems related to the concepts of statistical ward continuity and statistically ward compactness, ordinary continuity, ordinary compactness, and some other kinds of continuities. One may expect this concept to be a useful tool in the field of two normed  space theory in modelling various problems occurring in many areas of science, computer science, information theory and biological science. For a further study, we suggest to investigate quasi-Cauchy sequences of points, fuzzy functions and statistically ward continuity for the fuzzy functions in a $2$-normed fuzzy spaces. However due to the change in settings, the definitions and methods of proofs will not always be analogous to those of the present work. We note that the study in this paper can be carried to $n$-normed spaces without any difficulty (see for example \cite{ReddyandDuttaOnEquivalenceofnNorms}, and \cite{EsiandOzdemirIstronglysummablesequencespacesinnnormedspacesdefinedbyidealconvergenceandanOrliczfunction} for the definition of an $n$-normed space, \cite{SarabadanandTalebiStatisticalconvergenceandidealconvergenceofsequencesoffunctionsin$2$-normedspaces} and \cite{SahinerGurdalYigitTIdealconvergencecharacterizationofthecompletionoflinearnnormedspaces}).

\end{document}